\documentclass[11pt]{amsart}

\usepackage{amsthm, amsfonts, graphicx, pinlabel}
\usepackage[all]{xy}
\usepackage{hyperref}
\usepackage[usenames,dvipsnames]{xcolor}

\usepackage{url}
\usepackage{mathtools}

\newcommand{\df}{\ensuremath{\partial}}

\def\TM+{T^*(\rr_+ \times M)}

\newcommand{\rr}{\ensuremath{\mathbb{R}}}
\newcommand{\zz}{\ensuremath{\mathbb{Z}}}

\theoremstyle{plain}
\newtheorem{thm}{Theorem}[section]
\newtheorem{cor}[thm]{Corollary}
\newtheorem{lem}[thm]{Lemma}

\newtheorem{prop}[thm]{Proposition}

\theoremstyle{definition}
\newtheorem{defn}[thm]{Definition}

\theoremstyle{remark}
\newtheorem{rem}[thm]{Remark}

\numberwithin{equation}{section}  

\newcommand{\dfn}[1]{{\textbf{#1}}}

\newcommand{\alg}{\ensuremath{\mathcal{A}}}
\newcommand{\aug}{\ensuremath{\varepsilon}}
\newcommand{\leg}{\ensuremath{\Lambda}}

\DeclareMathOperator{\img}{im}

\begin{document}

\title{Lagrangian Fillings of Legendrian $4$-Plat Knots} 

\author[E. Lipman]{Erin R. Lipman} \address{University of Chicago, Chicago, IL 60637} \email{erlipman@uchicago.edu} 

\author[J. Sabloff]{Joshua M. Sabloff} \address{Haverford College,
Haverford, PA 19041} \email{jsabloff@haverford.edu} \thanks{JMS and ERL were
partially supported by NSF grant DMS-1406093 in the preparation of this paper.}

\begin{abstract}
	We  characterize which Legendrian $4$-plat knots in the standard contact $3$-space have exact orientable Lagrangian fillings.  As a corollary, we show that the underlying smooth knot types of fillable Legendrian $4$-plats are positive.
\end{abstract}

\date{\today}

\maketitle

\section{Introduction}
\label{sec:intro}

Questions designed to probe the relationship between properties of smooth knots and those of their Legendrian representatives have driven the field of Legendrian knot theory from its inception.  Bennequin's seminal paper \cite{bennequin}, for example, revealed a connection between the Seifert genus of a smooth knot and the maximal Thurston-Bennequin number of its Legendrian representatives.  Finding upper bounds for the Thurston-Bennequin number from  properties of the underlying smooth knot has provided many interesting connections between smooth and Legendrian knot theory through the slice  genus \cite{rudolph},  quantum invariants \cite{fw:braids-jones,morton-short:homfly-article,rudolph:kauffman-bound}, Khovanov homology \cite{lenny:khovanov, plamenevskaya:transverse-Kh, shumakovitch}, and Heegaard-Floer homology \cite{plamenevskaya:tau}, among others.

Chantraine deepened the connection between Legendrian knots and the slice genus by proving that an exact, orientable Lagrangian filling of a Legendrian knot realizes the slice genus of the underlying smooth knot \cite{chantraine}. To simplify language, we say that a Legendrian knot is \dfn{fillable} if it has an exact, orientable Lagrangian filling; see Section~\ref{sec:lagr-cob} for a precise definition. Chantraine's result motivates the central questions for this paper:

\begin{quote}
Which smooth knot types have fillable Legendrian representatives?  Which Legendrian representatives are fillable?
\end{quote}

It is known, for example, that all positive knots have fillable Legendrian representatives \cite{positivity}. Using a result of Boileau and Orevkov \cite{bo:qp} about symplectic fillings of transverse knots, we see that a necessary condition for smooth knot type to have a fillable Legendrian representative is for the knot to be quasipositive; see \cite{polyfillability} and \cite{positivity} for more details.

When asking about the fillability of a particular Legendrian knot, a general geometric picture has yet to take shape (though see \cite{nrssz:aug-sheaf}). Hence, it is useful to approach the problem by working with families of Legendrian knots. In this paper, we investigate Legendrian $4$-plats, which have previously been considered by Ng \cite{lenny:2-bridge}, who proved that every $2$-bridge knot has a Legendrian $4$-plat representative, and by Casey and Henry \cite{casey-henry}, who proved that Legendrian $4$-plats have at most one Chekanov polynomial. Our main result is a complete characterization of fillable Legendrian $4$-plats; see Section~\ref{ssec:2-bridge} for a specification of the band terminology used in the statement of the theorem:

\begin{thm} \label{thm:main}
A Legendrian $4$-plat knot is fillable if and only if each negative band has at most two crossings and each internal band has at least two crossings.
\end{thm}

We can relate this theorem to the aforementioned result about positive knots as follows:

\begin{cor} \label{cor:pos}
If a Legendrian $4$-plat knot is fillable, then its underlying smooth knot type is positive.
\end{cor}

In particular, we see that in order to investigate smooth knot types that are fillable but not positive, we must look beyond Legendrian $4$-plats.

To prove these results, we organize the remainder of the paper as follows:  after setting notation for Legendrian $4$-plats and the non-classical ruling and Legendrian contact homology invariants in Section~\ref{sec:Legendrian Knots}, we discuss Lagrangian fillings and obstructions to their existence from the aforementioned invariants in Section~\ref{sec:filling}.  The tools developed in Sections~\ref{sec:Legendrian Knots} and   \ref{sec:filling} are brought to bear on the proofs of Theorem~\ref{thm:main} and Corollary~\ref{cor:pos} in Sections~\ref{sec:Proof} and \ref{sec:La}, respectively.


\section{Legendrian Knots and their Invariants}
\label{sec:Legendrian Knots}

In this section, we describe Legendrian $4$-plat knots and then briefly discuss non-classical invariants of Legendrian knots, especially rulings and Legendrian contact homology.  These invariants will give rise to obstructions to the existence of a Lagangian filling, as we shall describe in Section~\ref{sec:filling}.  The goal of the present section is not to give a comprehensive exposition, but rather to recall standard ideas and to set notation.  We assume familiarity with basic notions of Legendrian knots, their front diagrams, and their classical invariants in the standard contact $\rr^3$; see \cite{etnyre:knot-intro}, for example.

\subsection{Legendrian $4$-Plat Knots}
\label{ssec:2-bridge}

Our main result pertains to a family of Legendrian knots that we will call Legendrian $4$-plat knots.

\begin{defn} \label{defn:leg-2-bridge}
A \dfn{Legendrian $4$-plat} is a Legendrian link admitting a front diagram with exactly four cusps (two left and two right).
\end{defn}

Figure~\ref{fig:2bridgeex} gives an example of a Legendrian $4$-plat knot.  The knot in the figure is in \dfn{plat form}, in which all of the left cusps have the same $x$-coordinate, all of the right cusps have the same or arbitrarily close $x$ coordinate,  and no two crossings have the same $x$-coordinate.  Any Legendrian front can be put in plat form by pulling each of the left (resp. right) cusps outwards to match the $x$-coordinate of the leftmost (resp. rightmost) cusp using Reidemeister II moves.

One more basic term to introduce is a \dfn{spanning arc}, which is a smooth path in a front diagram joining a left to a right cusp. Thus, in terms of spanning arcs, a Legendrian $4$-plat knot is a Legendrian knot admitting a front with exactly four spanning arcs. 

\begin{figure}
\centerline{
\includegraphics[width=8cm]{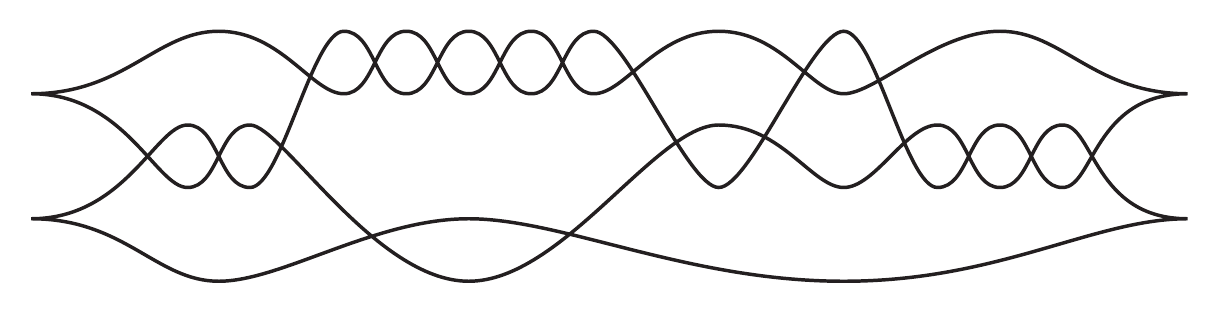}}
\caption{A Legendrian $4$-plat knot represented by tuple $[3,(6,2),2,(2,0),4]$.}
\label{fig:2bridgeex}
\end{figure}

We note that the smooth knot type of a Legendrian $4$-plat knot is $2$-bridge. In fact, Ng shows in \cite{lenny:2-bridge} that any smooth $2$-bridge knot has a Legendrian $4$-plat representative. Further, that representative may be assumed to have a spanning arc without crossings. On the other hand, not all Legendrian representatives of $2$-bridge knots have Legendrian $4$-plat diagrams. For example, one can show that certain (tb-maximal) Legendrian representatives of sufficiently large twist knots investigated by Chekanov \cite{chv} are obstructed from having $4$-plat representatives by the Maslov indices of their linearized Legendrian contact homologies. 

We will now introduce the notation that we will use to discuss Legendrian $4$-plat knots in the proof of the main theorem. A tb-maximal $4$-plat Legendrian front in plat form necessarily has its first and last crossings between the center two strands. Thus a potentially fillable Legendrian knot in $4$-plat form is of the form shown in Figure \ref{fig:cannonical form}. We can represent such a front with an $n$-tuple (where $n$ is necessarily odd) of the form
\[[b_1, b_2=(b_{2u},b_{2l}), b_3,\cdots,b_{n-1}=(b_{(n-1)u}, b_{(n-1)l}),b_n],\]
where $b_i$, $b_{iu}$ and $b_{il}$ are the numbers of crossings in the corresponding boxes in the figure. To ensure that a front is represented by a unique tuple, we require that $b_i>0$ for odd $i$ and that $b_i=b_{iu}+b_{il}>0$ for even $i$. The front in Figure~\ref{fig:2bridgeex} is represented by the tuple $[3,(6,2),2,(2,0),4]$.

Throughout the proof, we will abuse notation and use $b_i$ to refer both to the band itself and the number of the crossings in the band.  We will refer to $b_i$ as the $i^{th}$ \dfn{band} and will call it positive or negative depending on the sign of the crossings it contains (it is clear that all crossings within a single band have the same sign). The $b_i$ with even $i$ will be referred to as \dfn{side} bands and the $b_i$ with odd $i$ will be referred to as \dfn{center} bands. The term \dfn{sub-band} will denote a center band or one (possibly trivial) portion ($b_{iu}$ or $b_{il}$) of a side band. A side band is called \dfn{split} if both $b_{iu}$ and $b_{il}$ are non-zero. A band $b_i$ with $1<i<n$ is called an \dfn{internal} band.

\begin{figure}

\labellist
	\small
	\pinlabel $b_1$ [b] at 69 45
	\pinlabel $b_{2u}$ [b] at 122 72
	\pinlabel $b_{2l}$ [b] at 122 16
	\pinlabel $b_3$ [b] at 177 45
	\pinlabel $b_{(n-1)u}$ [b] at 315 72
	\pinlabel $b_{(n-1)l}$ [b] at 315 16
	\pinlabel $b_n$ [b] at 383 45
\endlabellist

\centerline{\includegraphics[width=4.5in]{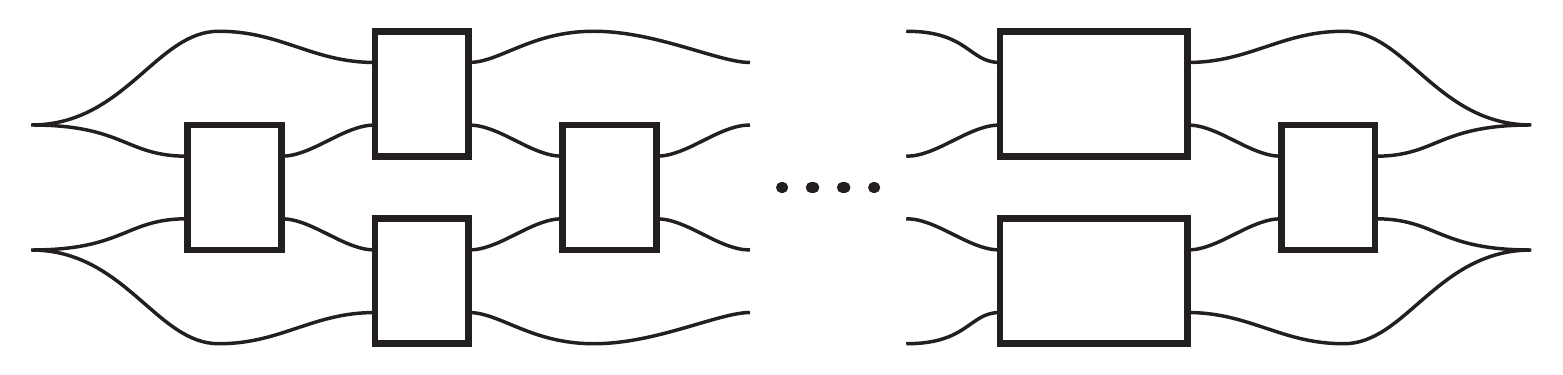}}
\caption{General form of a Legendrian $4$-plat knot in plat form.}
\label{fig:cannonical form}
\end{figure}
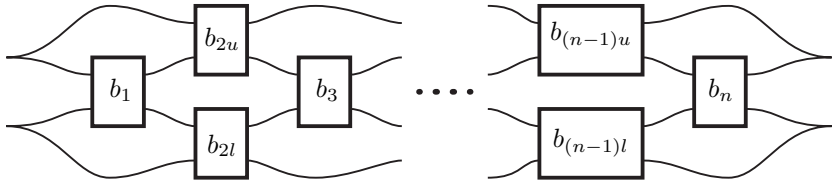

\subsection{Rulings}
\label{sec:rulings}

Now that we have established precise notation for describing Legendrian $4$-plat knots, we may begin our discussion of non-classical invariants of Legendrian knots. The first non-classical invariant we use to obstruct the existence of a Lagrangian filling is the number of \dfn{normal rulings}. The non-existence of a normal ruling will serve as a useful obstruction to the existence of a  Lagrangian filling.

Our definition of a ruling follows that of \cite{f-r}. Generically, all crossings in a front have distinct $x$ coordinates, and no crossing has the same $x$-coordinate as a cusp. We assume this is the case in our discussion of rulings.

\begin{defn}[\cite{chv-pushkar,fuchs:augmentations}]
Let $\leg$ be a front diagram of a Legendrian knot. A \dfn{ruling} of $\leg$ consists of the following:
\begin{enumerate}
\item A pairing of left and right cusps.
\item For each pair $l$ and $r$ of corresponding left and right cusps, a pair of piecewise smooth paths in $\leg$, disjoint except at the cusps, joining $l$ and $r$. We require that each path have strictly increasing $x$ coordinate, and that paths joining distinct pairs of cusps intersect only at crossings. The union of the two paths joining a pair of cusps bounds a \dfn{ruling disk}.
\end{enumerate}
\end{defn}

At each crossing, paths joining two distinct pairs of cusps either cross transversely or switch from one spanning arc to another. The latter type of crossing is called a \dfn{switch}. 

We  require that our rulings are \dfn{normal}. The normality condition refers to the relative heights of strands from the involved pairs of paths at switched crossings. Figure \ref{fig:normality} shows the allowed configurations of strands near switched crossings, which are exactly those configurations in which ruling disks are either disjoint or nested near a switch (this is where the requirement that crossings have distinct $x$ coordinates becomes important). Figure~\ref{fig:normality}(b) shows a normal ruling for a front of the trefoil knot.

\begin{figure}
\centerline{\includegraphics[width=4.5in]{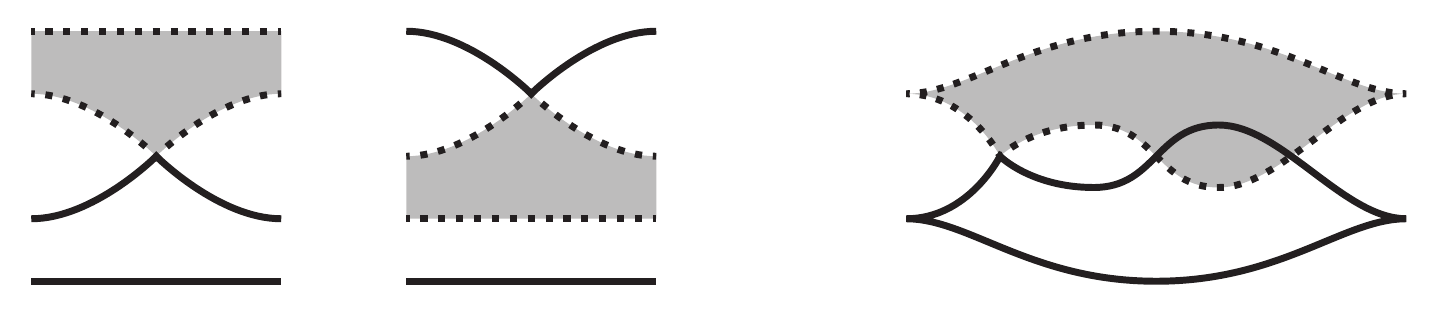}}
\caption{(a) Acceptable configurations for relative heights of strands around a switched crossing in a normal ruling; reflections about the horizontal are also acceptable. (b) A normal ruling of a Legendrian trefoil.}
\label{fig:normality}
\end{figure}

This construction is useful in Legendrian knot theory because of the following theorem:

\begin{thm}[\cite{chv-pushkar}]
The number of  normal rulings of a front is invariant under Legendrian isotopy.
\end{thm}

In our formulation, the number of normal rulings is, in fact, an invariant of the underlying smooth knot:  Rutherford proved this by showing that a refined count of the number of  normal rulings can be read off of the Kauffman polynomial \cite{rutherford:kauffman}.

\subsection{Legendrian Contact Homology}
\label{ssec:LLCH}

The second non-classical invariant we will use to obstruct the existence of a Lagrangian filling is the linearized Legendrian contact homology.  We will begin with Ng's combinatorial formulation of the Chekanov-Eliashberg differential graded algebra (DGA) on front diagrams \cite{lenny:computable}; we will also draw on \cite{hr:dga}. Later on, in Section~\ref{sec:seidel}, we will use the fact that this invariant arises from the theory of $J$-holomorphic curves when we connect it to Lagrangian fillings via the Seidel isomorphism.  Note that we will only define the DGA for a front in plat form, and our gradings will be reduced modulo $2$ from the beginning.

\subsubsection{The Chekanov-Eliashberg DGA}

Given a front for Legendrian knot $\leg$, we label the crossings and right cusps with labels $a_1, \dots, a_n$. Let $A$ be the vector space generated by the labels $a_1, \ldots, a_n$ over $\zz/2$.  We turn $A$ into a $\zz/2$-graded vector space by assigning grading $0$ to all labels of positive crossings and grading $1$ to all labels of negative crossings and right cusps. The algebra $\alg$ is defined to be the unital tensor algebra over $A$:
\[ \alg = \bigoplus_{k=0}^\infty A^{\otimes k}.\]
Note that we may think of the elements of $\alg$ as linear combinations of words in the labels $a_i$.

The final step in the definition of the Chekanov-Eliashberg DGA is the differential $\df$. The computation of the differential is based on counting the \dfn{admissible disks} in the front originating at each of the generating crossings and cusps. 

\begin{defn}
For a front $\leg$ in plat form, an \dfn{admissible disk} is the image of an embedding of disk $D^2$ into the $xz$ plane having the following properties:
\begin{enumerate}
    \item The image of the boundary $\partial D^2$ is in $\leg$.
    \item The restriction of the map to $\partial D^2$ is smooth away from crossings and cusps.
    \item The map at crossings and cusps looks like one of the  configurations shown in Figure~\ref{fig:singularities}.
    \item The map has one originating singularity (at a crossing or right cusp) and one terminating singularity (at a left cusp); see Figure~\ref{fig:singularities}.
\end{enumerate}
\end{defn}

\begin{figure}
\centerline{
\includegraphics[width=4in]{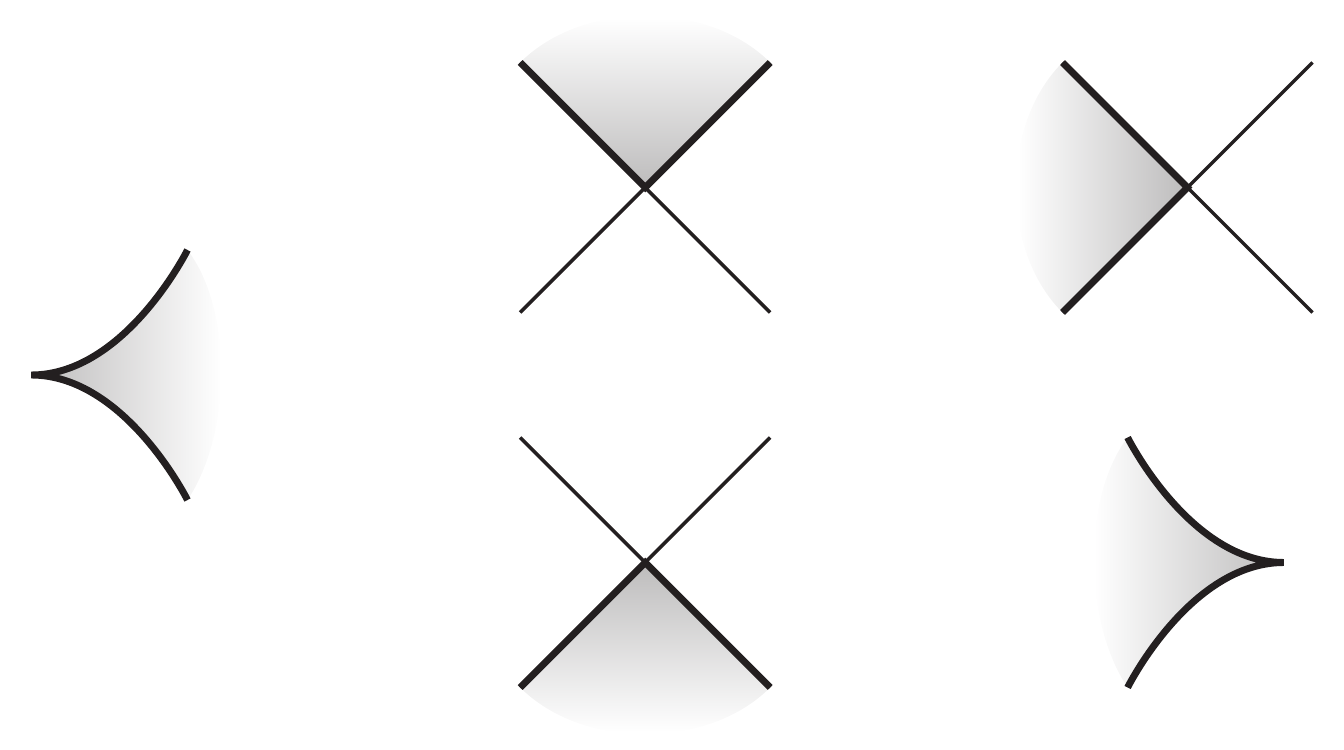}}
\caption{Allowed singularities of admissible disks: (left) terminating singularities, (center) negative corners, and (right) originating singularities or positive corners.}
\label{fig:singularities}
\end{figure}

Originating singularities are called \dfn{positive corners} of the disk and singularities at other crossings (as in figure~\ref{fig:singularities}) are called \dfn{negative corners}; we warn the reader not to confuse the sign of a corner of a disk with the sign of the underlying crossing. For any admissible disk, the originating singularity is the rightmost point of the disk and the terminating singularity is the leftmost point. Figure~\ref{fig:some_disks} shows two admissible disks in a front for the $5_2$ knot.

\begin{figure}
\centering
\includegraphics[width=4in]{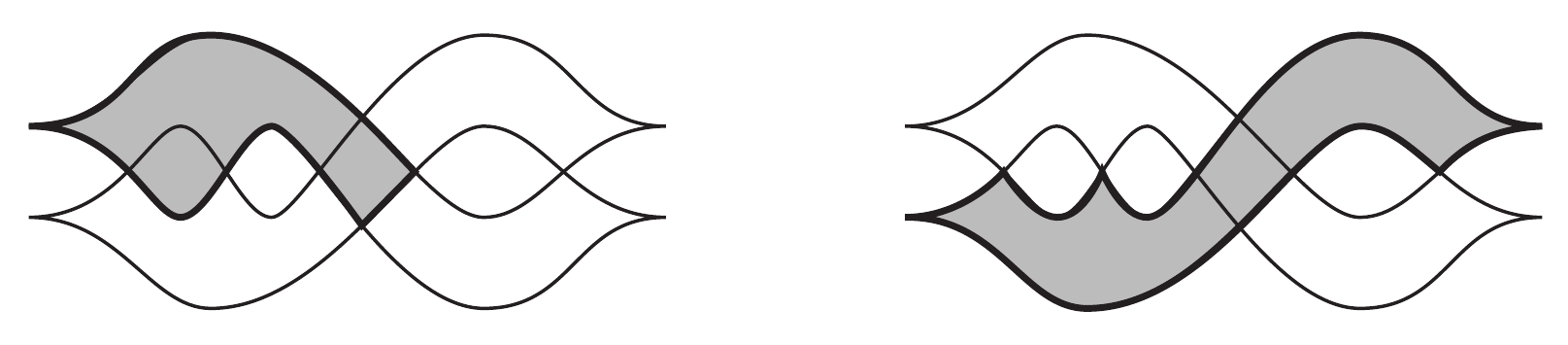}
\caption{Two admissible disks for a Legendrian $5_2$ knot.}
\label{fig:some_disks}
\end{figure}

To each admissible disk, we may associate a word $w\in \alg$ consisting of the labels of the negative corners ordered counterclockwise. A disk with no negative corners is represented by the trivial word $1$. Let $\Delta(a,w)$ be the set of admissible disks, up to reparametrization, originating at $a$ with negative corners described by $w$ in counterclockwise order starting at $a$. Let $\#\Delta(a,w)$ be the count of elements in this class modulo $2$. We can now define the differential.

\begin{defn}
Define the differential $\df: \alg \to \alg$ on a generator $a$ by
\[
  \df a =\begin{cases}
               \sum_w  \#\Delta(a,w)w & \text{if a is crossing}\\
               1 + \sum_w \#\Delta(a,w)w & \text{if a is right cusp}\\
            \end{cases}
\]
Extend $\df$ to all of $\alg$ by linearity and the Leibniz rule.
\end{defn}

That the map $\df$ is a genuine graded differential was proven by Chekanov \cite{chv} and adapted to front diagrams by Ng \cite{lenny:computable}.

\subsubsection{Linearized Legendrian Contact Homology}

While the homology of the full Chekanov-Eliashberg DGA $(\alg, \df)$ is invariant under Legendrian isotopy, its non-linear nature means that it is difficult to extract practical invariants. Chekanov's linearization technique is one important way to reduce the complexity of the DGA; see \cite{chv} for the original idea, and \cite{products,lenny:computable} for alternative expositions.

The key idea in the linearization process is to use an augmentation of $(\alg, \df)$, i.e. a DGA morphism $\aug: (\alg, \df) \to (\zz/2,0)$.  It is not always the case that there is an augmentation for the DGA \cite{fuchs:augmentations,fuchs-ishk,rulings}, but as we will note in Section~\ref{sec:seidel} below,  the Chekanov-Eliashberg DGA of any fillable Legendrian knot has an augmentation.

Given an augmentation $\aug$, we define a change of coordinates $\phi^\aug(a)=a+\aug(a)$ and a new differential on $\alg$ by
\[\df^\aug=\phi^\aug\df(\phi^\aug)^{-1} = \phi^\aug \df.\]
We define a map $\df_1^\aug: A \to A$ by restricting the domain and codomain of $\df^\aug$. It will prove useful to characterize disks that contribute to $\df_1^\aug$. If a word $a_1 \cdots a_n \in \alg$ appears in $\df a$, then the associated disk contributes $a_i$ to $\df^\aug_1 a$ if and only if $\aug(a_j) = 1$ for all $j \neq i$. We will refer to such a disk as an \dfn{augmented disk}. An augmented disk necessarily has at most one negative corner $a_i$ at a negative crossing.

It is straightforward to verify that $\df_1^\aug$ is a graded differential on the graded vector space $A$.  We define the \dfn{linearized Legendrian Contact Homology} by 
\[LCH^\aug_k(\leg) = H_k(A, \df^\aug_1).\]

\begin{thm}[\cite{chv}]
	The set of all linearized Legendrian contact homology groups taken over all possible augmentations of $(\alg, \df)$ is invariant under Legendrian isotopy.
\end{thm}

\section{Lagrangian Fillings of Legendrian Knots}
\label{sec:filling}

As our goal in this paper is to understand when a Legendrian $4$-plat knot has a Lagrangian filling, we now expand the class of objects under consideration from Legendrian knots to Lagrangian cobordisms between them.  In this section, we review the definition of a Lagrangian cobordism and a combinatorial construction of such a cobordism, and then describe obstructions to the existance of a Lagrangian cobordism from the empty set to a Legendrian knot.

\subsection{Lagrangian Cobordisms and Fillings}
\label{sec:lagr-cob}

We fix a contact manifold $(M, \alpha)$ and its symplectization $(\rr \times M, d(e^t\alpha))$, where $t$ is the $\rr$ coordinate.  Note that the symplectization is an exact symplectic manifold with primitive $e^t \alpha$.  

\begin{defn} \label{defn:cobordism}
A submanifold $L \subset (\rr \times M, d(e^t \alpha))$  is an \dfn{(exact, orientable) Lagrangian cobordism} from a Legendrian link $\leg_-$ to another Legendrian link $\leg_+$ if $L$ is an exact orientable Lagrangian submanifold,   there exists a pair of real numbers $r_{\pm}$ such that
\begin{align*}
L\cap ((-\infty,r_-] \times M) &= (-\infty,r_-] \times \leg_- \\
L\cap ([r_+, \infty) \times M) &= [r_+, \infty) \times \leg_+
\end{align*}
and the primitive of $e^t\alpha$ along $L$ is constant for $t < r_-$ and for $t > r_+$. 
\end{defn} 

Note that the final condition is automatically satisfied if the ends $\leg_\pm$ are connected; see \cite{chantraine:disconnected-ends}. We will drop the descriptors ``exact'' and  ``orientable'' in subsequent discussions, though we assume that they still hold.

With the notion of Lagrangian cobordism in hand, we define the central object for this paper.

\begin{defn} \label{defn:filling}
A \dfn{Lagrangian filling} of a Legendrian knot $\leg$ is a Lagrangian cobordism from the empty set to $\leg$.  A Legendrian knot is \dfn{Lagrangian fillable} if it has a Lagrangian filling.
\end{defn}

\subsection{Constructions of Lagrangian Cobordisms}

The following theorem allows us to construct Lagrangian fillings combinatorially. The statement is adapted from Theorem 4.2 of \cite{bst:construct}; similar constructions appear in \cite{rizell:surgery, ehk:leg-knot-lagr-cob}.

\begin{figure}

\labellist
	\large
	\pinlabel $\emptyset$ [b] at 32 25

	\small
	\pinlabel $0$-handle [r] at 52 90
	\pinlabel $1$-handle [r] at 235 90
\endlabellist

\centerline{\includegraphics[height=2in]{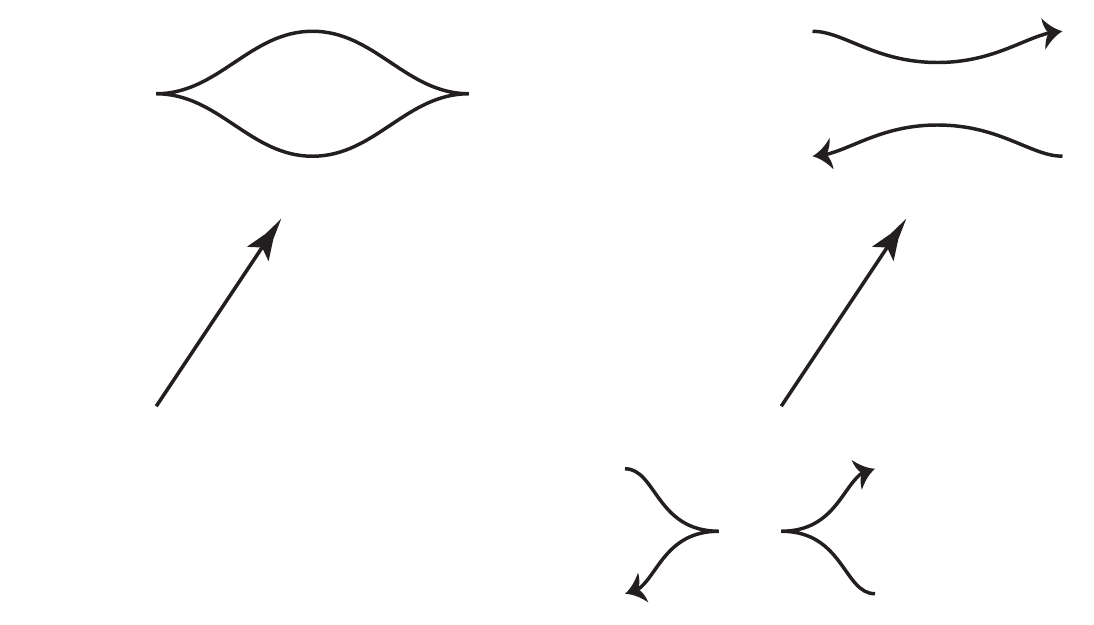}}
\caption{Attachment of a Legendrian $0$-handle (left) and $1$-handle (right). }
\label{fig:handles}
\end{figure}

\begin{thm}[\cite{bst:construct,rizell:surgery, ehk:leg-knot-lagr-cob}]
\label{thm:constructible}
If two Legendrian knots $\leg_-$ and $\leg_+$ in $(\rr^3, \xi_0)$ are related by any of the following moves, then there exists a Lagrangian cobordism from $\leg_-$ to $\leg_+$.
\begin{description}
\item[Isotopy] $\leg_+$ and $\leg_-$ are Legendrian isotopic.
\item[$0$-handle attachment] $\leg_+$ is equivalent to the disjoint union of $\leg_-$ and a Legendrian unknot as in the left side of Figure~\ref{fig:handles}.
\item[$1$-handle attachment] $\leg_+$ is related to $\leg_-$ by the attachment of a $1$-handle as in the right side of Figure~\ref{fig:handles}. Note the compatible orientations of the two cusps.
\end{description}
\end{thm}

See Figure~\ref{fig:filling-ex} for an example of a Lagrangian filling of a Legendrian $5_2$ knot using Theorem~\ref{thm:constructible}.

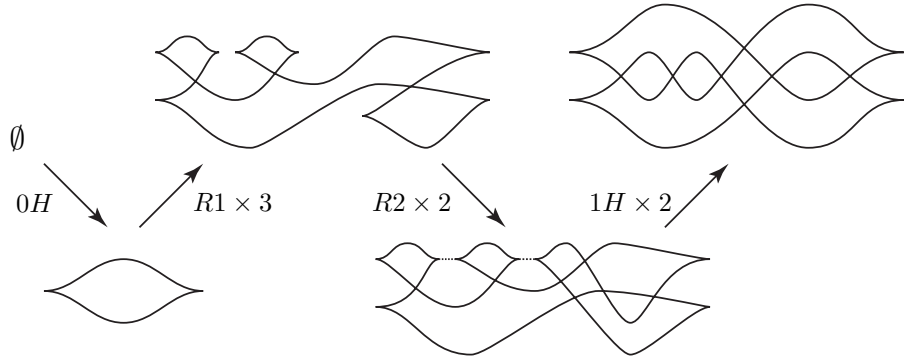
\begin{figure}
\labellist
	\large
	\pinlabel $\emptyset$ [b] at 30 120

	\small
	\pinlabel $0H$ [r] at 55 95
	\pinlabel ${R1 \times 3}$ [l] at 125 95
	\pinlabel ${R2 \times 2}$ [r] at 280 95
	\pinlabel ${1H \times 2}$ [r] at 405 95
\endlabellist

\centerline{\includegraphics[width=5in]{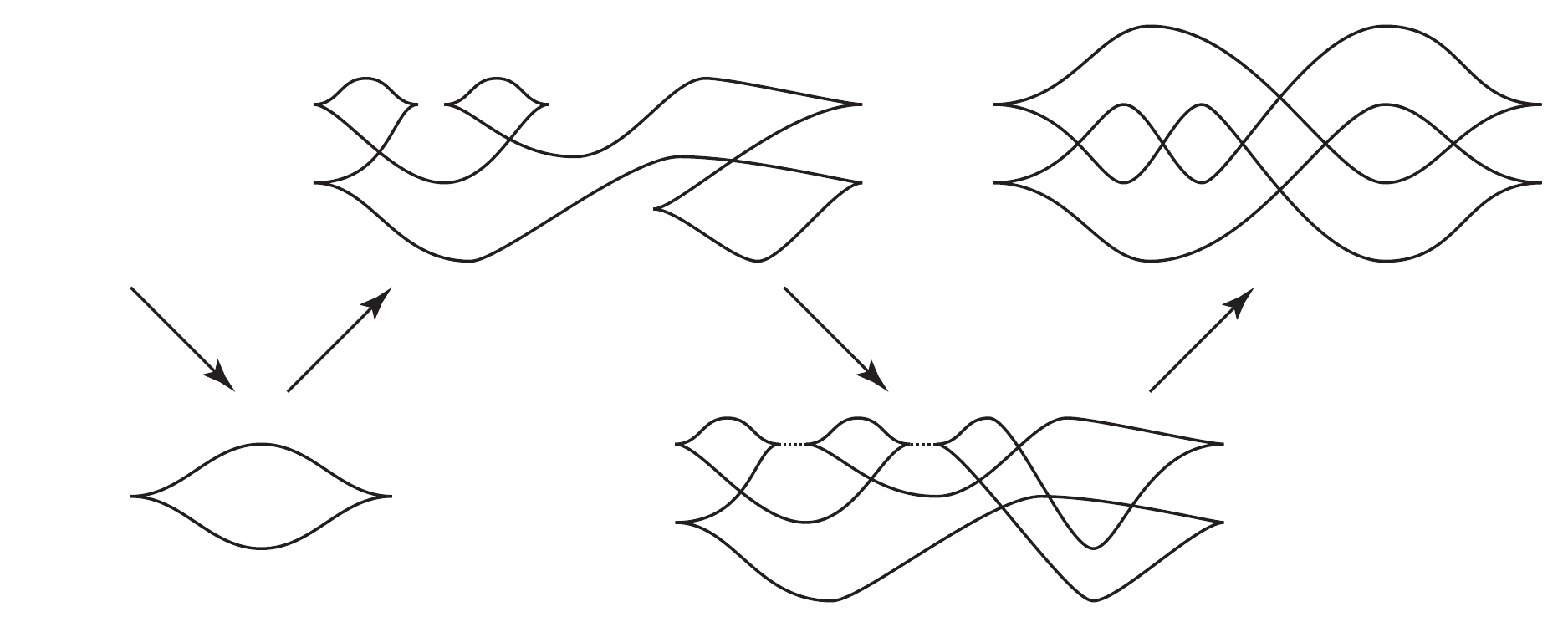}}
\caption{A Lagrangian filling of a Legendrian $5_2$ knot.}
\label{fig:filling-ex}
\end{figure}

\subsection{Obstructions to Lagrangian Fillings}
\label{sec:seidel}

In our characterization of fillable Legendrian $4$-plat knots, we will need computable obstructions to the existence of a filling. To begin, the classical invariants each give rise to easy-to-state obstructions. 

\begin{prop}[\cite{chantraine}] \label{prop:classical-obstr}
	If a Legendrian knot $\leg$ has a Lagrangian filling, then:
	\begin{enumerate}
		\item $r(\leg) = 0$ and
		\item $tb(\leg)$ is maximal.
	\end{enumerate}
\end{prop}

The non-classical invariants discussed above also yield obstructions to the existence of a Lagrangian filling.  The ruling obstruction that we need is particularly simple to state; the proof appears in \cite{positivity}:

\begin{prop}  \label{prop:ruling-obstr}
	If a Legendrian knot has a Lagrangian filling, then it has a  normal ruling.
\end{prop}

The deepest of the obstructions we require comes from the linearized Legendrian contact homology. That there is a connection between the Chekanov-Eliashberg DGA and Lagrangian cobordisms has long been known, at least since the seminal work in \cite{egh}.  More concretely, Ekholm proved in \cite{ekholm:lagr-cob} that a Lagrangian filling of a Legendrian knot induces an augmentation for the knot's Chekanov-Eliashberg DGA (here, we take gradings modulo $2$ as in Section~\ref{ssec:LLCH}).   We will make use of an isomorphism between the linearized Legendrian contact homology with respect to the augmentation induced by an exact filling $L$ and the singular homology of $L$. This isomorphism was conjectured by Seidel and was proven in \cite{rizell:lifting} and \cite{ekholm:lagr-cob}. We reduce the generality of the Seidel Isomorphism to Legendrian knots in $(\rr^3, \alpha_0)$ with gradings taken modulo $2$.

\begin{thm}[Seidel Isomorphism \cite{rizell:lifting,ekholm:lagr-cob}]
\label{thrm: Siedel}
Let $\leg$ be a Legendrian knot in the standard contact $\rr^3$ with a Lagrangian filling $L$ in the symplectization $\rr \times \rr^3$.  If $\aug$ is the augmentation induced by $L$, then there is an isomorphism
\[ LCH^\aug_k(\leg) \simeq H_{k+1}(L,[r_+, \infty) \times \leg).\]
\end{thm}

As any orientable filling is homeomorphic to a punctured $n$-holed torus, the Seidel isomorphism has the following consequence: if $\leg$ is fillable, then there exists an augmentation $\aug$ such that $\leg$ has the following linearized Legendrian contact homology (again, the grading is taken modulo $2$):
\begin{enumerate}
\item $\dim LCH^\aug_1 (\leg) =1$ and
\item $\dim LCH^\aug_0 (\leg) =2n$ (for some $n\in\mathbb{N}$).
\end{enumerate}

We can characterize the generator of $LCH^\aug_1(\leg)$ using the fundamental class of \cite[\S5]{duality}.  If a front diagram for (fillable) $\leg$ has exactly two right cusps with labels $a$ and $a'$, then the generator of $LCH^\aug_1(\leg)$ is represented by a cycle of the form
\[a+a'+x,\]
where $x$ lies in the subspace of $A$ generated by the crossings of the front diagram. Further, it is evident from the definition of the differential $\df$ that the right cusps $a$ and $a'$ cannot be boundaries.

We are led to the following corollary of the Seidel isomorphism and the existence of the fundamental class, which will play a key role in the proof of the main theorem:
 
\begin{lem}
\label{lem: fundamental class}
If for all augmentations of the Chekanov-Eliashberg DGA of $\leg$, there exists a negative crossing that is a cycle but not a boundary, then $\leg$ is not fillable.
\end{lem}

\section{Proof of the Main Theorem}
\label{sec:Proof}

In this section, we prove the main theorem of the paper, namely that a Legendrian $4$-plat knot is fillable if and only if each negative band has at most two crossings and each internal band has at least two crossings. Such knots, as the proof will show, have the property that their bands alternate in sign (see Corollary~\ref{cor: alternating bands}) and that if the bands at the ends are negative, then they each have one crossing.

The proof proceeds by winnowing down the possible forms of negative bands in a fillable Legendrian $4$-plat front, until only knots of the desired form remain. We then show that a knot in this form is fillable.

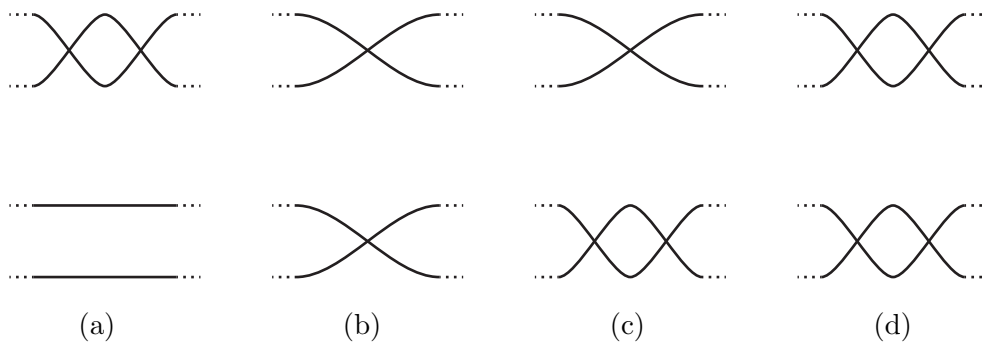
\begin{figure}
\labellist
	\pinlabel (a) [b] at 42 0
	\pinlabel (b) [b] at 142 0
	\pinlabel (c) [b] at 242 0
	\pinlabel (d) [b] at 342 0
\endlabellist

\centerline{\includegraphics{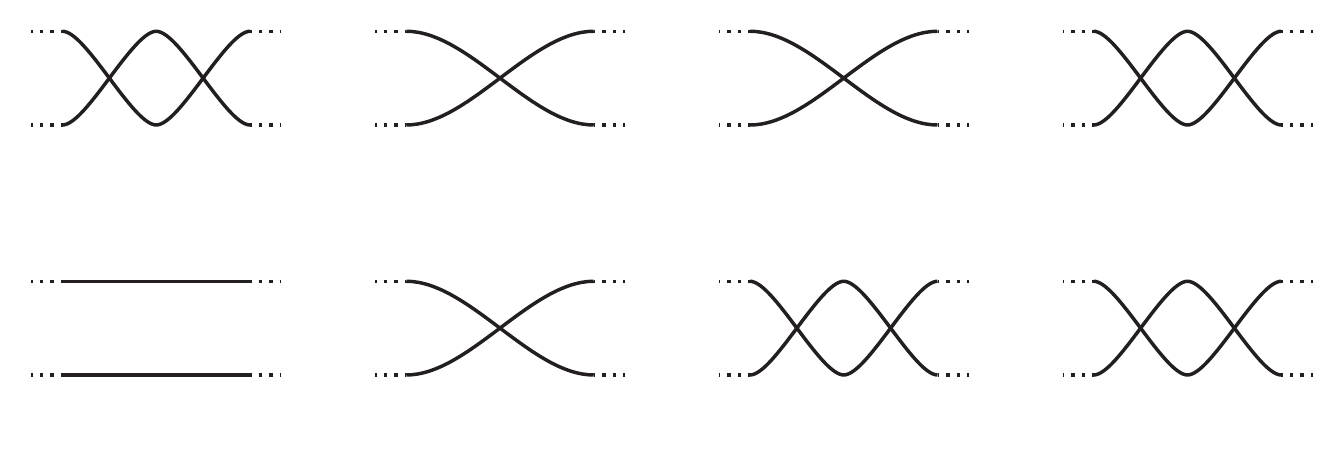}}
\caption{Side bands, excluding single and long bands: (a) double, (b) split double, (c) split triple, and (d) split quadruple.}
\label{fig:bands}
\end{figure}

We first create a taxonomy for describing bands by their number of crossings and the distribution of these crossings across sub-bands:
\begin{enumerate}
    \item A \dfn{single} band is a band with only one crossing.
    \item A \dfn{long} sub-band is a sub-band with three or more crossings. A long band is a band containing a long sub-band.
    \item A \dfn{double} band is a band with exactly two crossings. A \dfn{split double} is a split band with one crossing in each sub-band; see Figure~\ref{fig:bands}(a,b).
    \item A \dfn{split triple} is a split band with two crossings in one sub-band and a single crossing in the other; see Figure~\ref{fig:bands}(c).
    \item A \dfn{split quadruple} is a split band with two crossings in each sub-band; see Figure~\ref{fig:bands}(d).

\end{enumerate}

In Sections~\ref{sec:Ut} through \ref{sec:Fa}, we show that a fillable Legendrian $4$-plat knot  cannot contain a single internal band, a negative long band, a negative split triple, or a negative split quadruple, respectively. A knot not subject to any of the above obstructions is of the form specified by the main theorem. In Section~\ref{sec:So} we prove that a knot of this form is indeed fillable.

\subsection{Eliminating Internal Single Bands}
\label{sec:Ut}

We begin the winnowing process by eliminating internal single bands.

\begin{lem}
\label{lem: Ut}
A Legendrian $4$-plat knot with an internal single band (either positive or negative) is not fillable.
\end{lem}

\begin{figure}
\labellist
	\pinlabel $A$ [b] at 45 0
	\pinlabel $B$ [b] at 185 0
\endlabellist

\centerline{\includegraphics[height=1.5in]{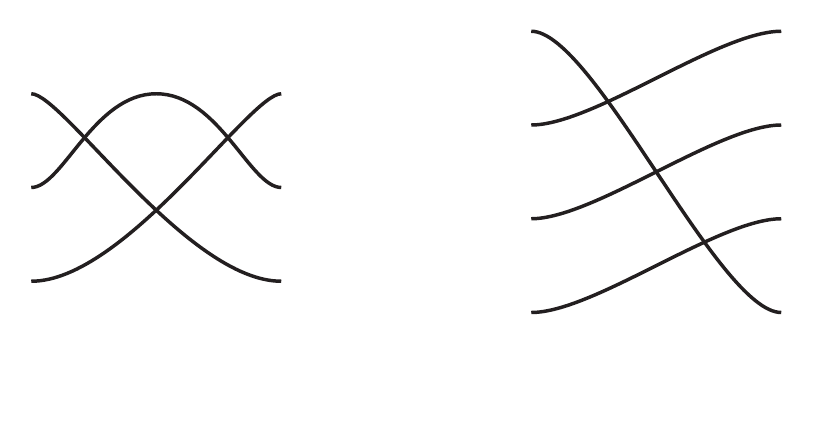}}
\caption{Configurations for internal bands with single crossing. For the configuration on the left, the remaining strand may be either above or below the portion of the knot shown in the diagram.}
\label{fig:singlesa}
\end{figure}

\begin{figure}
\labellist
	\pinlabel $A$ [b] at 45 0
	\pinlabel $B$ [b] at 185 0
\endlabellist

\centerline{\includegraphics[height=1.5in]{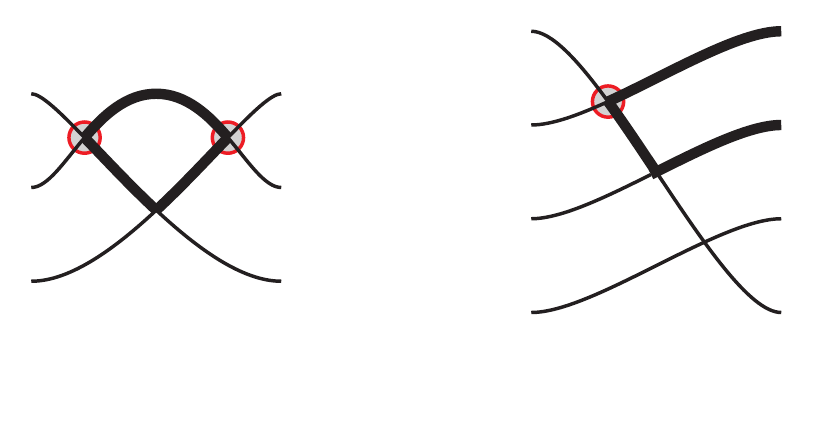}}
\caption{Attempt at a ruling near single crossing with the crossing switched.}
\label{fig:singlesb}
\end{figure}

\begin{figure}
\labellist
	\pinlabel $A$ [b] at 45 0
	\pinlabel $B$ [b] at 185 0
	\pinlabel ? [b] at 28 65
\endlabellist

\centerline{\includegraphics[height=1.5in]{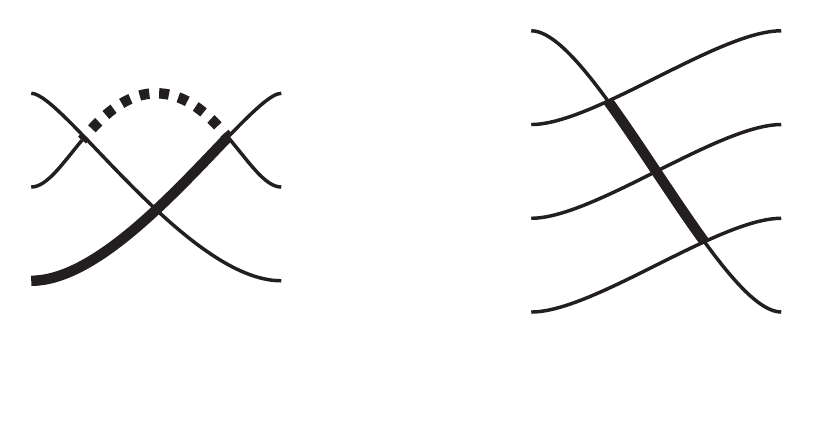}}
\caption{Attempt at a ruling near single crossing with the crossing not switched. }
\label{fig:singlesc}
\end{figure}

\begin{proof}
Suppose that $\leg$ is fillable and contains an internal single band $b$. We will derive a contradiction to Proposition~\ref{prop:ruling-obstr} by showing that $\leg$ does not have a normal ruling. To set up the proof, note that near the band $b$, the front must look like one of the diagrams in Figure~\ref{fig:singlesa} or its reflection over a horizontal line. We refer to the configurations on the left and right of the figure as $A$ and $B$, respectively. 

First suppose that there is a switch at the  crossing of the band $b$, as in Figure~\ref{fig:singlesb}. In configuration $A$, the normality condition implies that the upper ruling path at the switch must belong to the same ruling disk as the strand immediately above it (recall that a $4$-plat knot has only two ruling disks). Both of these strands are thickened in the diagram on the left side of the figure. The two thickened strands must intersect each other, which contradicts our definition of ruling disks. An identical argument can be carried out for configuration $B$, as illustrated in the image on the right side of the figure. Thus, neither configuration permits a ruling where the crossing of $b$ is switched.

Next suppose that the crossing of $b$ is not switched, as in Figure~\ref{fig:singlesc}. In configuration $A$, the thickened strand and the dotted strand must belong to distinct ruling disks since if they belonged to a single disk, then this disk would have a self-intersection at the right-hand crossing. The strand marked with a question mark, then, cannot belong to either the thickened or dotted disk, since it intersects strands from each of the disks transversely. This is a contradiction since a ruling of a $4$-plat knot has only two ruling disks. Similarly, in configuration $B$, Three of the four strands at the right of the diagram cannot be in the same ruling disk as the thickened strand. This is a contradiction since only two of these strands can belong to the other ruling disk. Thus neither configuration permits a ruling where the crossings of $b$ are not switched.
\end{proof}

\subsection{Eliminating Negative Long Bands}
\label{sec:Re}

The next step in the process is to eliminate the possibility of a negative long band.  

\begin{lem}
\label{lem: Re}
A Legendrian $4$-plat knot with a negative long band is not fillable.
\end{lem}

\begin{proof} Suppose that the front of a fillable Legendrian $4$-plat $\leg$ contains a negative long band as in Figure~\ref{fig:long bands}. We will derive a contradiction by showing that the hypotheses of Lemma~\ref{lem: fundamental class} hold:  for any augmentation of the Chekanov-Eliashberg DGA of $\leg$, there exists an odd degree (i.e.\ negative) crossing that is a cycle but not a boundary.

We briefly set notation.  Let $\aug$ be an augmentation. Let $a$ denote the rightmost crossing in the long negative band and let $a'$ denote the crossing that lies third from the right of the band. 

Clearly, there are no admissible disks originating at $a$, so $a$ is a cycle. To show that $\leg$ is not fillable using Lemma~\ref{lem: fundamental class}, we need to show that $a$ is not a boundary.  In fact, we show that for any generator $b$, its linearized differential is of the form 
\begin{equation}
\label{eq: augmented differential}
\df_1^\aug b = c(a+a')+x,
\end{equation}
where $c \in \zz/2$ and $x$ lies in the subspace of $A$ generated by all crossings except $a$ and $a'$.  

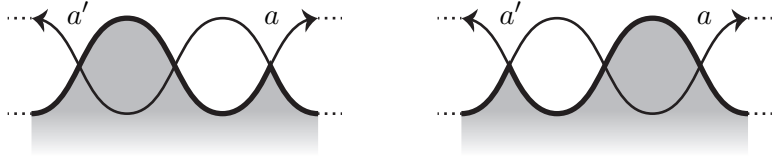
\begin{figure}
\labellist
	\pinlabel $a'$ [b] at 37 58
	\pinlabel $a$ [b] at 110 58
	\pinlabel $a'$ [b] at 200 58
	\pinlabel $a$ [b] at 273 58
\endlabellist

\centerline{\includegraphics{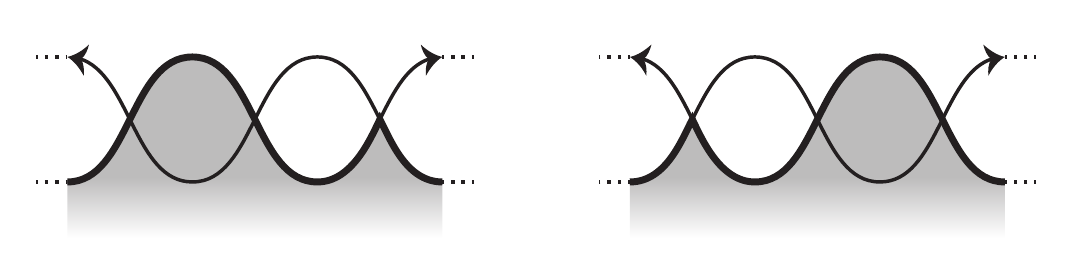}}
\caption{A disk with a corner at $a$ and the corresponding disk with a corner at $a'$. A reflection of this picture about a horizontal line is also possible.}
\label{fig:long bands}
\end{figure}

To see that Equation~(\ref{eq: augmented differential}) holds, simply note that augmented disks with a corner at $a$ or at $a'$ cannot have another corner in the long negative band.  Thus, Figure~\ref{fig:long bands} demonstrates that there is a one-to-one correspondence between augmented disks with a corner at $a$ and those with a corner at $a'$; Equation~(\ref{eq: augmented differential}), and hence the lemma, follows.
\end{proof}

\subsection{Eliminating Negative Split Triples}
\label{sec:Mi}

Combining Lemmas \ref{lem: Ut} and \ref{lem: Re}, we have now established that every internal negative central band of a fillable Legendrian $4$-plat knot has exactly two crossings, and also that a negative internal band with an odd number of crossings is necessarily a split triple. We now eliminate the possibility of split triples. 

\begin{lem}
\label{lem: Mi}
A Legendrian $4$-plat knot with a negative split triple is not fillable.
\end{lem}

Our proof of Lemma~\ref{lem: Mi} requires somewhat more work than those for the previous two lemmas.  We do obtain a side benefit from the extra work, however, as we will  also obtain the following result concerning the global structure of fillable $4$-plat knots.

\begin{cor}
\label{cor: alternating bands}
If $\leg$ is a fillable Legendrian $4$-plat knot, then the bands of $\leg$ strictly alternate in sign. 
\end{cor}

The proof will appear at the end of the section.

The following two combinatorial lemmas give us insight into the global structure of $4$-plat knots and are key ingredients in the proof of Lemma~\ref{lem: Mi}.

\begin{prop}
If $b_i$ is a positive band of a Legendrian $4$-plat, then $b_{i+1}$ is a negative band.
\label{lem: Fred}
\end{prop}

\begin{figure}
\centerline{\includegraphics{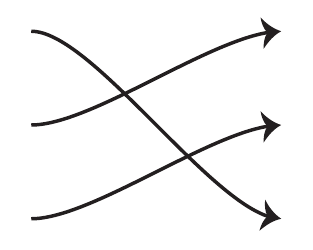}}
\caption{Crossings from two adjacent positive bands. }
\label{fig:double positive}
\end{figure}

\begin{proof}
Suppose not, i.e.\ that a Legendrian $4$-plat has two adjacent positive bands. Such a knot will look locally like the diagram in Figure~\ref{fig:double positive}, up to reflection through the horizontal or the vertical.  Notice that this diagram necessarily has three spanning arcs oriented in the same direction, which contradicts the fact that a $4$-plat knot necessarily has exactly two spanning arcs oriented in each direction.
\end{proof}

\begin{prop}
If $b_i$ is an internal negative band of a Legendrian $4$-plat, then $b_{i+1}$ has the same sign as $b_{i-1}$ if and only if $b_i$ has an even number of crossings.
\label{lem: George}
\end{prop}

\begin{proof}
To analyze changes in sign as we pass from one side of a negative band to the other, we think of the band as permuting the order of orientations of spanning arcs. We denote the orientations of the spanning arcs at a fixed $x$ coordinate using a word in the letters $L$, for an arc oriented to the left, and $R$, for a spanning arc oriented to the right.  The letters in the word record the orientations from top to bottom.  Such a word necessarily has two of each letter.  Denote by $W_l$ (resp. $W_r$) the orientation word for the strands to the left (resp. right) of the band $b_i$.

It is easy to determine the sign of a crossing between two adjacent spanning arcs:  it is positive if the letters representing orientations agree, and negative otherwise. There are three classes of orientation words:
\begin{enumerate}
\item $O_1 = \{LLRR, RRLL\}$
\item $O_2 = \{LRRL, RLLR\}$
\item $O_3 = \{LRLR, RLRL\}$
\end{enumerate}
Each of the words in a given class of orderings yields the same signs for corresponding crossings.  Further, for all of the classes, all crossings in a given split side band have the same sign.

We now use the notation built up above to prove the proposition. If the band $b_{i}$ has an even number of crossings, then either the order of the orientations to the right of $b_{i}$ is the same as the order to the left or --- if $b_{i}$ is a split band with an odd number of crossings in both sub-bands --- all orientations are reversed.  Thus, the words $W_l$ and $W_r$ belong to the same class of orderings and hence the signs of $b_{i-1}$ and $b_{i+1}$ must be the same.

Now suppose that $b_i$ has an odd number of crossings. If $b_i$ is a center band, then $W_l$ must belong to $O_1$ or $O_3$.  Since moving across $b_i$ transposes the orientations of the two center arcs, we see that $W_r$ belongs to $O_3$ or $O_1$, respectively.  By inspection of the words in $O_1$ and $O_3$, the signs of the side bands $b_{i-1}$ and $b_{i+1}$ are opposite.

If $b_i$ is a side band, then $W_l$ must belong to $O_2$ or $O_3$. Since moving across $b_i$ transposes the orientations of two of the side arcs, we see that $W_r$ belongs to $O_3$ or $O_2$, respectively.  By inspection of the words in $O_2$ and $O_3$, the signs of the center bands $b_{i-1}$ and $b_{i+1}$ are opposite.
\end{proof}

These propositions will allow us to make use of the following local obstruction to fillability.

\begin{prop}
If $\leg$ contains negative bands $b_{i-1}$ and $b_i$ such that $b_i$ is a split triple, then $\leg$ is not fillable.
\label{lem: config}
\end{prop}

\begin{proof}
Locally at bands $b_{i-1}$ and $b_i$, $\leg$ looks like the diagram in Figure \ref{fig:split triple} or its reflection through the horizontal. Observe that there is no augmented disk in $\leg$, originating at a positive crossing, with a negative corner at the crossing marked $a$, as any augmented disk with a negative corner at $a$ must originate at $b$ or $c$. Thus $a$ is a negative crossing that is a cycle but not a boundary, so $\leg$ cannot be fillable by Lemma~\ref{lem: fundamental class}.
\end{proof}

\begin{figure}
\labellist
	\small
	\pinlabel $a$ [t] at 55 48
	\pinlabel $b$ [b] at 60 72
	\pinlabel $c$ [t] at 80 25
\endlabellist

\centerline{\includegraphics{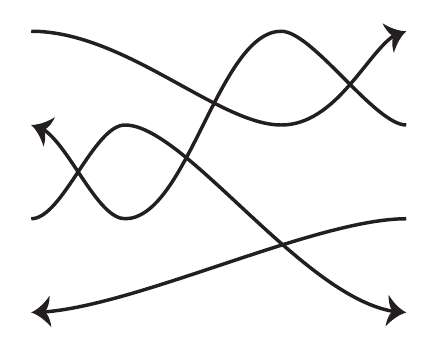}}
\caption{Negative split triple preceded by negative band.}
\label{fig:split triple}
\end{figure}

To prove Lemma~\ref{lem: Mi}, it remains to show that the presence of a split triple band necessarily induces the configuration described in Proposition~\ref{lem: config}. We will use the following useful observation that is a direct consequence of Propositions \ref{lem: Fred} and \ref{lem: George}.

\begin{prop}
\label{prop: alternating}
If $\leg$ contains a sequence of bands $b_i,\dots,b_k$ such that $b_i$ is positive and none of the bands $b_i,\dots,b_{k-1}$ are odd negative bands, then $b_i,\dots, b_k$ strictly alternate in sign.
\end{prop}

\begin{proof}[Proof of Lemma~\ref{lem: Mi}]
Let $\leg$ be a Legendrian $4$-plat knot, without internal single bands or long negative bands, given by the tuple $[b_1, b_2,\dots,  b_n]$. Suppose that $\leg$ contains a negative split triple and that $b_k$ is the leftmost negative split triple.  This setup implies that all negative bands to the left of $b_k$ have an even number of crossings.  Further, we know that $k$ must be even, or else $b_k$ would be a center band. There are three cases to consider, based on the signs of $b_1$ and $b_2$.

\begin{description}
\item[Case 1] Suppose that $b_1$ is a negative band and $b_2$ is a positive band. Then $b_2,\dots,b_k$ satisfy the hypotheses of Proposition~\ref{prop: alternating}, so this sequence of bands is strictly alternating in sign. This means that $k$ is odd, which contradicts the fact noted above that $k$ must be even.

\item[Case 2]
Suppose that $b_1$ and $b_2$ are both negative bands. Proposition~\ref{lem: George} implies that each band $b_i$ with $i<k$ is negative. In particular, $b_{k-1}$ is a negative center band. Thus we have the configuration described in Proposition~\ref{lem: config}, so $\leg$ is not fillable.

\item[Case 3]
Suppose that $b_1$ is a positive band. By Proposition~\ref{prop: alternating}, the first $k$ bands alternate in sign, and in particular $b_{k-1}$ is positive. If  $b_k$ were the only split triple in $\leg$, Proposition~\ref{lem: George} would imply that all bands to the right of $b_k$, and in particular $b_n$, are negative. The facts that $b_1$ is positive and $b_n$ is negative together imply that $r(\leg)\neq 0$, which is an obstruction to fillability. Thus, it must be the case  that there is another split triple $b_{k'}$. Proposition~\ref{lem: George} then implies that all bands $b_k,\dots,b_{k'}$ are negative, so $b_{k'-1}$ is negative. Thus, Proposition~\ref{lem: config} tells us that $\leg$ is not fillable.
\end{description}
\end{proof}

\begin{proof}[Proof of Corollary~\ref{cor: alternating bands}]
	Lemmas~\ref{lem: Ut}, \ref{lem: Re}, and \ref{lem: Mi} imply that every negative band in a fillable Legendrian $4$-plat has an even number of crossings.  Thus, by Proposition~\ref{prop: alternating}, all $4$-plats satisfying the hypotheses of Cases $1$ and $3$ in the proof above must have bands of alternating sign.  
	
A $4$-plat knot satisfying the hypotheses of Case $2$ must have only negative bands and have an even number of crossings in each internal band.  The external bands must both have even numbers of crossings as well.  To see why, note that if the first band is negative and has but one crossing, then the $4$-plat must lie in Case 1.  If the first band has two crossings, then, in the notation of the proof of Lemma~\ref{lem: George}, the orientation words of the spanning arcs will be of class $O_3$.  Thus, the last band must also have two crossings, or else the orientation word just to the left of the right cusps will be in class $O_1$, which is impossible.  Thus, in Case $2$, we see that \emph{all} bands have exactly two crossings. Such a $4$-plat is necessarily a $2$-component link, not a knot, so the Corollary follows.
\end{proof}

\begin{rem} \label{rmk:1-neg-xing}
	The proof of Corollary~\ref{cor: alternating bands} shows that if the first band of a fillable $4$-plat knot is negative, then it must have exactly one crossing; the same is true for the last band.
\end{rem}

\subsection{Eliminating Negative Split Quadruples}
\label{sec:Fa}

We now know that in a fillable $4$-plat knot, bands strictly alternate in sign, internal negative center bands have exactly two crossings, and negative side bands have either two or four crossings.  Having proved Lemmas~\ref{lem: Ut} through \ref{lem: Mi}, we need only to eliminate negative split quadruples.

\begin{lem}
\label{lem: Fa}
A Legendrian $4$-plat knot with a negative split quadruple is not fillable.
\end{lem}

\begin{proof}
Suppose not, i.e.\ that there is a fillable Legendrian $4$-plat knot with a negative split quadruple.  By Corollary~\ref{cor: alternating bands}, we may assume that $b_1$ is a positive band. Let $b_k$ denote the leftmost negative split quadruple. 

\begin{figure}
\labellist
	\pinlabel $+$ [b] at 181 53
	\pinlabel $+$ [b] at 288 53
	
	\pinlabel $+$ [b] at 181 190
	\pinlabel $+$ [b] at 288 190
	\pinlabel $+$ [b] at 72 190
	
	\pinlabel $-$ [b] at 126 217
	\pinlabel $-$ [b] at 126 163

	\pinlabel $a$ [b] at 373 90
	\pinlabel $a'$ [t] at 374 33
	\pinlabel $b$ [b] at 391 65
	
	\pinlabel $a$ [b] at 373 225
	\pinlabel $a'$ [t] at 374 168
	\pinlabel $b$ [b] at 391 202
\endlabellist

\centerline{\includegraphics[width=5in]{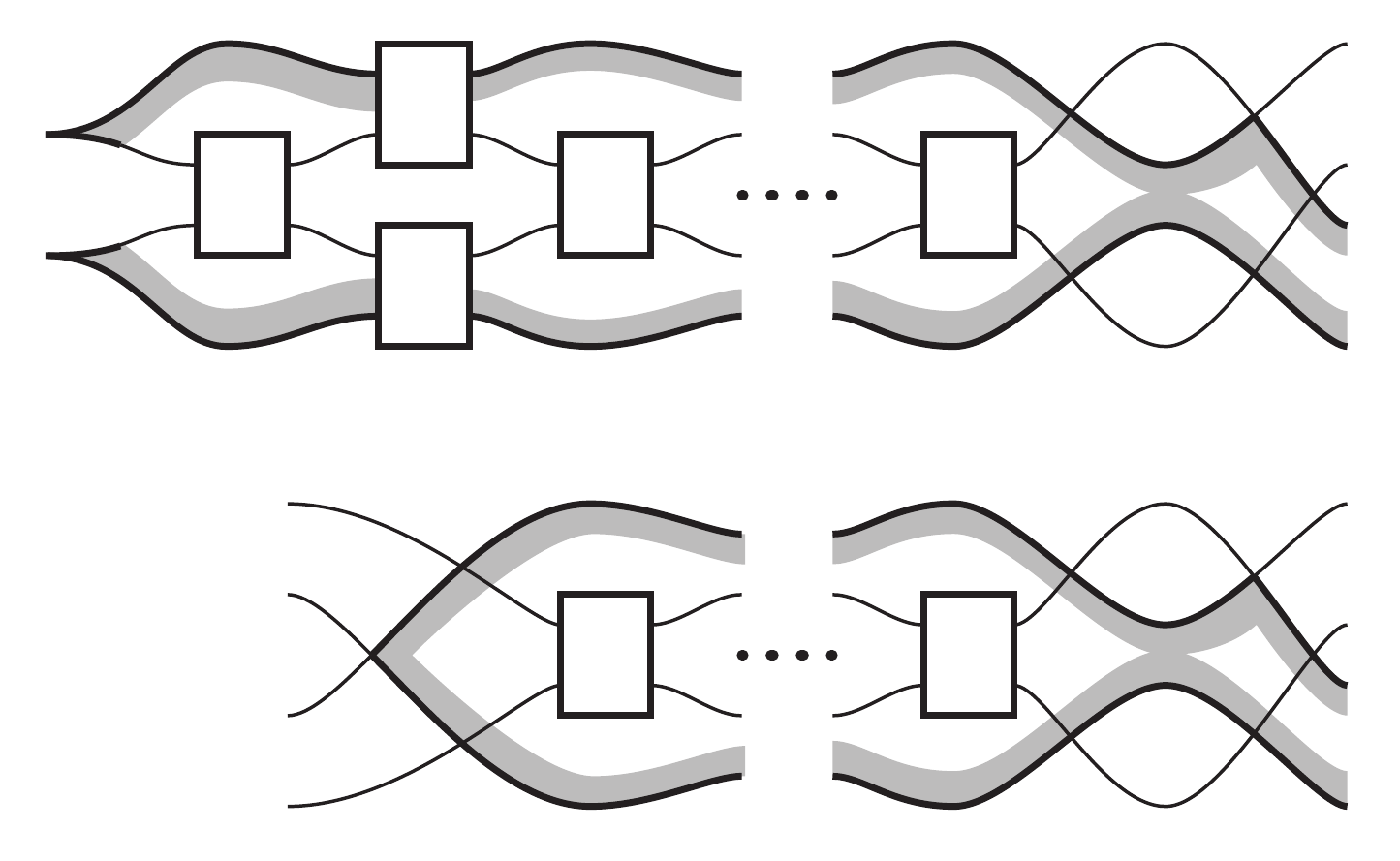}}
\caption{Attempt at constructing an augmented disk with a negative corner at $a$ for link without (top) and with (bottom) a negative split double to the left of the leftmost negative split quadruple band.}
\label{fig:Lupin}
\end{figure}
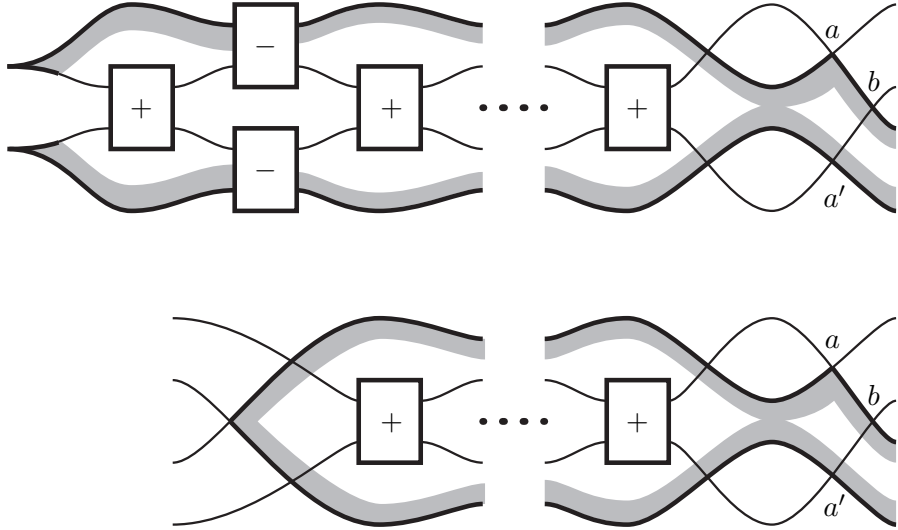

Consider first the case where there is no split double band to the left of $b_k$. In this case, the knot is of the form shown at the top of Figure~\ref{fig:Lupin}, with crossings $a$, $a'$, and $b$ as labeled. We attempt to construct an augmented disk, not originating at $b$, with $a$ as a negative corner, as pictured in Figure~\ref{fig:Lupin}. There is no valid augmented disk encompassing the union of these paths, since an admissible disk may not meet multiple left cusps. Thus any disk with $a$ as its only negative corner must originate at $b$. A symmetric argument gives the same result for the crossing $a'$.

Now consider the case where there \textit{is} a split double band to the left of $b_k$. In this case, the knot is of the form shown at the bottom of Figure~\ref{fig:Lupin} (the left-most split band in the figure is the split double). We attempt to construct an augmented disk, not originating at $b$, with $a$ as a negative corner, as pictured in Figure~\ref{fig:Lupin}. There is no valid augmented disk encompassing the union of these paths since the top and bottom strand of a disk may not intersect as they do to the left of the split double. Thus, as in the first case, an augmented disk with $a$ as a negative corner must originate at $b$, and a symmetric argument gives the same result for $a'$.

Since $a$ and $a'$ are each cycles, the subspace of $\ker \df^\aug_1$ generated by $a$ and $a'$ has dimension $2$. Since any augmented disk with a corner at $a$ or $a'$ originates at $b$, the subspace of $\img \df^\aug_1$ generated by $a$ and $a'$ has dimension at most $1$. It must be the case, then, that either $a$ or $a'$ is an odd degree cycle that is not a boundary. Thus, by Lemma~\ref{lem: fundamental class}, the knot cannot be fillable, a contradiction to the original hypothesis.
\end{proof}

Having shown that a fillable knot cannot contain a negative split quadruple, we have completed our proof that a fillable Legendrian $4$-plat knot in plat form is of the form described in the main theorem.

\subsection{Constructing Fillings}
\label{sec:So}

The final step in the proof of the main theorem is to show that a knot of the form described in the main theorem --- one in which negative bands have at most two crossings and internal negative bands have exactly two crossings --- is indeed fillable. In fact, our construction also works for links of the same form. 

The construction proceeds by induction on the number of center bands.  The key fact that drives the induction is encapsulated by the following lemma:

\begin{lem} \label{lem:good-pinch}
	If a $4$-plat link with at least three bands satisfies the hypotheses of the theorem, then the $4$-plat link formed by pinching off the first two bands and the first crossing of the third band (as in Figure~\ref{fig:pinch}) also satisfies the hypotheses of the theorem.
\end{lem}

\begin{figure}
\labellist
	\small
	\pinlabel $b_1$ [b] at 60 52
	\pinlabel $b_{2u}$ [b] at 113 79
	\pinlabel $b_{2l}$ [b] at 113 24
	\pinlabel $b_3-1$ [b] at 198 52
	\pinlabel $b_3-1$ [b] at 387 52
\endlabellist

\centerline{\includegraphics[width=4.5in]{pinch}}
\caption{Pinching off the first two bands and the first crossing of the third band to form a new $4$-plat link.}
\label{fig:pinch}
\end{figure}

The proof of the lemma is obvious, though some care must be taken with orientations on a case-by-case basis to guarantee that the orientations on the right side of the original $4$-plat agree with those on the new, pinched $4$-plat.

We now begin the proof proper.  Suppose that we have a $4$-plat link $\leg$ satisfying the hypotheses of the main theorem defined by the tuple $B=[b_1, b_2=(b_{2u}, b_{2l}), \dots, b_{2n}=(b_{(2n)u}, b_{(2n)l}), b_{2n+1}]$. The construction proceeds by induction on $n$ in two separate cases, depending on whether the leftmost band is positive or negative.

Let us begin with the case when the leftmost band of $\leg$ is positive.  For this case, we refer to Figure~\ref{fig:Ron}. The top diagram in the figure represents the case where $b_2$ is a (negative) split double and the bottom diagram represents the case where $b_2$ is a (negative) non-split double.

\begin{figure}
\labellist
	\small
	\pinlabel $b_1-1$ [b] at 72 49
	\pinlabel $b_1-1$ [b] at 72 183
	
	\pinlabel $b_1-1$ [b] at 235 49
	\pinlabel $b_1-1$ [b] at 235 183
	
	\pinlabel $b_3-1$ [b] at 415 49
	\pinlabel $b_3-1$ [b] at 415 183
\endlabellist

\centerline{\includegraphics[width=5in]{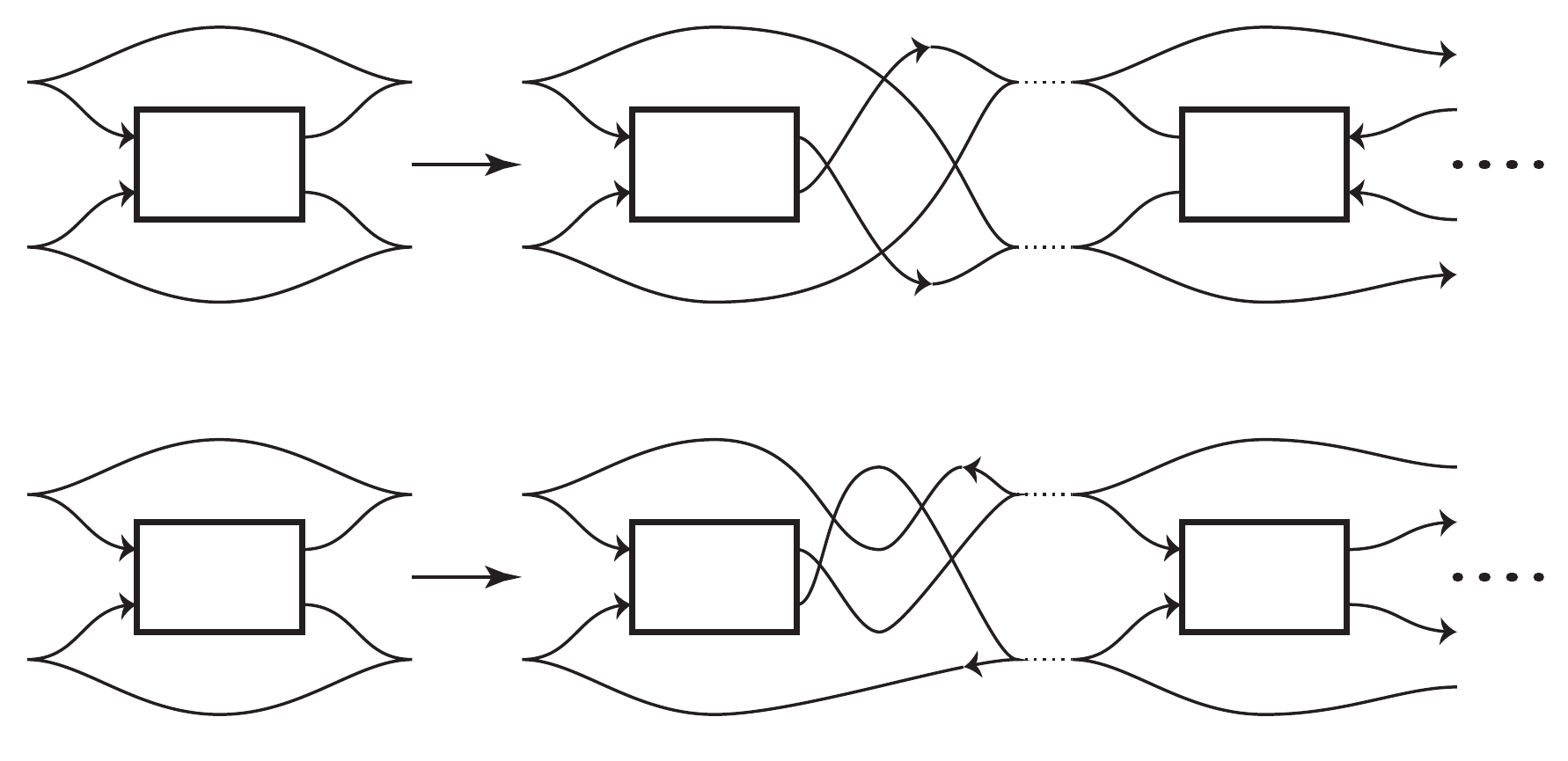}}
\caption{Inductive step for constructing a filling in the case where the initial band is positive. The top diagram represents the case where $b_2$ is a split double. The bottom diagram represents the case where $b_2$ is a non-split double.}
\label{fig:Ron}
\end{figure}
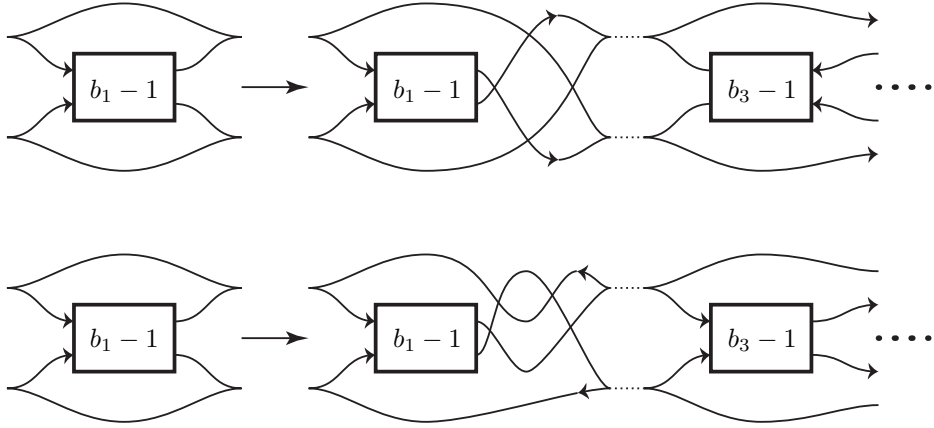

To see that $\leg$ is fillable, we follow the procedure of Figure~\ref{fig:Ron}.  On one hand, the links $\leg_0$ on the left side of Figure~\ref{fig:Ron} are positive, and hence fillable by \cite{positivity}. By an isotopy as in Figure~\ref{fig:Ron}, and using the fact that $b_2=2$ since it is an internal negative band, we see that $\leg_0$ may be defined by the tuple $B_0 = [b_1,b_2,1]$. On the other hand, the inductive hypothesis shows that the link $\leg'$ defined by the tuple $B' = [b_3-1, \ldots, b_{2n+1}]$ --- which satisfies the hypotheses of the theorem by Lemma~\ref{lem:good-pinch} --- is also fillable.  Thus, the disjoint union $\leg_0 \sqcup \leg'$ is fillable.  To complete the construction of a filling for the original link, we start with the filling of $\leg_0 \sqcup \leg'$ and add to it a cobordism that is constructed using a Legendrian isotopy of $\leg_0$, followed by the attachment of two $1$-handles as in Figure~\ref{fig:Ron}.

In the case where the leftmost band is negative, we recall Remark~\ref{rmk:1-neg-xing},  namely that a $4$-plat link in which negative bands have at most two crossings that begins with a negative band and has bands strictly alternating in sign must have a single crossing in its first and last band. With this observation in hand, the second case proceeds exactly as does the first, with the operations of Figure~\ref{fig:Hermione} taking the place of those in Figure~\ref{fig:Ron}.

\begin{figure}
\labellist
	\small
	\pinlabel $b_{2u}-2$ [b] at 64 72
	\pinlabel $b_{2l}$ [b] at 64 18

	\pinlabel $b_{2u}-2$ [b] at 235 72
	\pinlabel $b_{2l}$ [b] at 235 18
		
	\pinlabel $b_3-1$ [b] at 390 45
	\pinlabel $b_3-1$ [b] at 390 180
\endlabellist

\centerline{\includegraphics[width=5in]{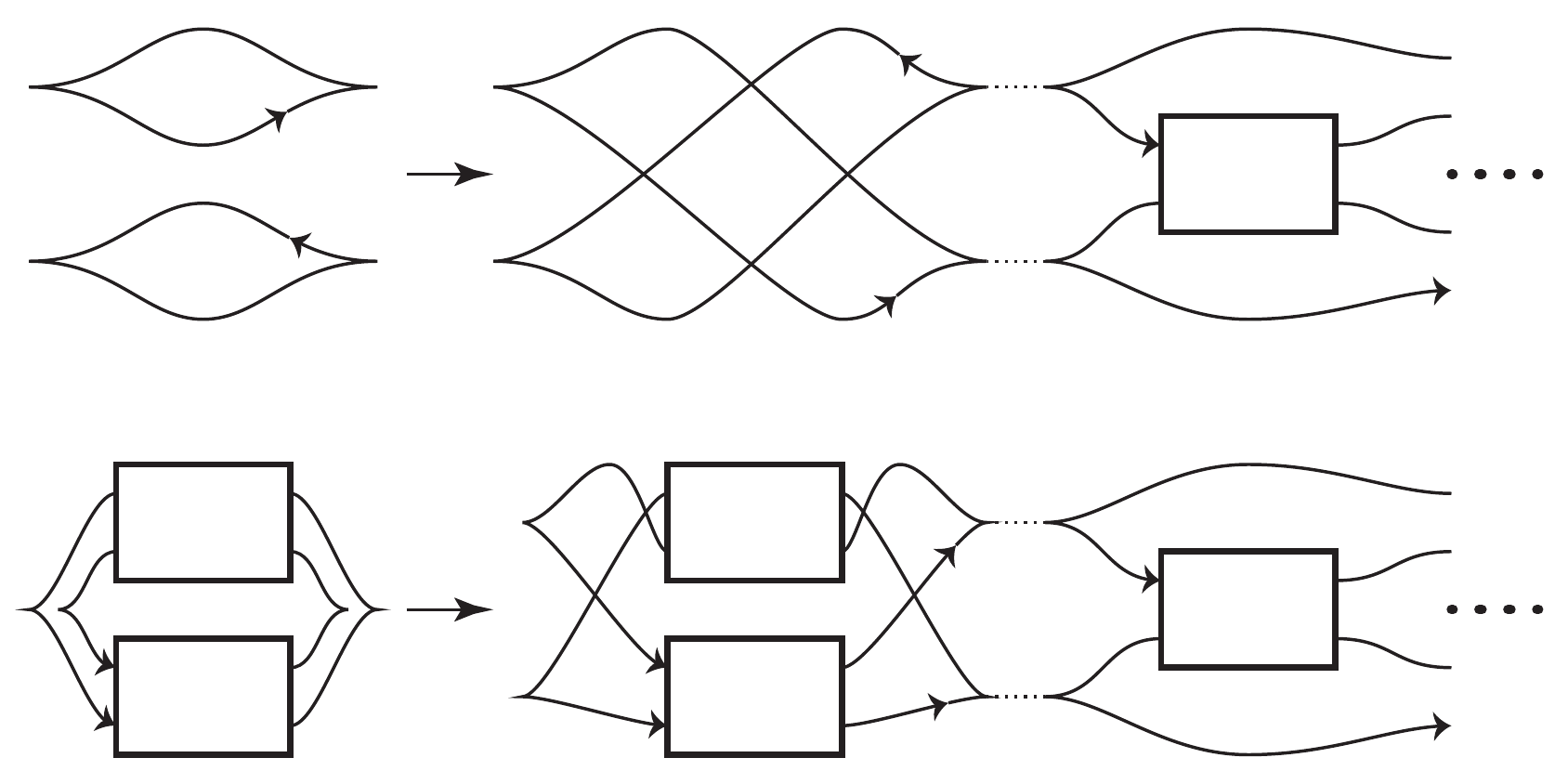}}
\caption{Inductive step for constructing a filling in the case where the initial band is negative. The top diagram represents the case where $b_2$ is a (positive) split double. The bottom diagram represents the case where one of the sub-bands of $b_2$ contains at least two crossings.}
\label{fig:Hermione}
\end{figure}
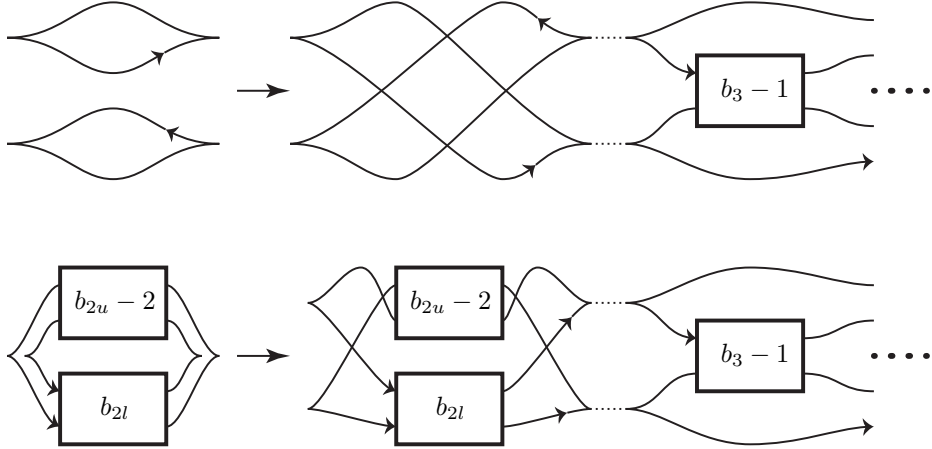

\section{Proof of Corollary~\ref{cor:pos}}
\label{sec:La}

In this final section, we prove Corollary~\ref{cor:pos}, which states that if a Legendrian $4$-plat is fillable, then the underlying smooth knot type is positive. Our proof proceeds inductively once again. Starting with the smooth version of a fillable Legendrian $4$-plat knot, we work from left to right on the knot diagram, using isotopy to eliminate negative crossings two at a time until only positive crossings remain. During the process of eliminating negative crossings, our knots will always be of the following form:

\begin{defn}
We will say that a diagram of a smooth knot is in \dfn{positive/$4$-plat (P4P)} form if there exist $x$ coordinates $x_\pm$ such that:
    \begin{enumerate}
        \item To the left of $x_-$, all crossings are positive. We will denote this portion of the knot by $P$.
        \item To the right of $x_+$, the diagram looks like the right-hand portion of the smoothing of a diagram of a Legendrian $4$-plat knot in plat form. We will denote this portion of the knot by $K$.
        \item Between $x_-$ and $x_+$, there are exactly four crossings, with signs (from left to right) positive, negative, negative, positive. We will call these the middle crossings.
    \end{enumerate}
\end{defn}

As a first step, suppose that the first band $b_1$ is negative.  As noted in Section~\ref{sec:So}, this band has but one crossing, which can be removed by a Reidemeister II move. We are now in a situation where the leftmost crossing of the diagram is positive.

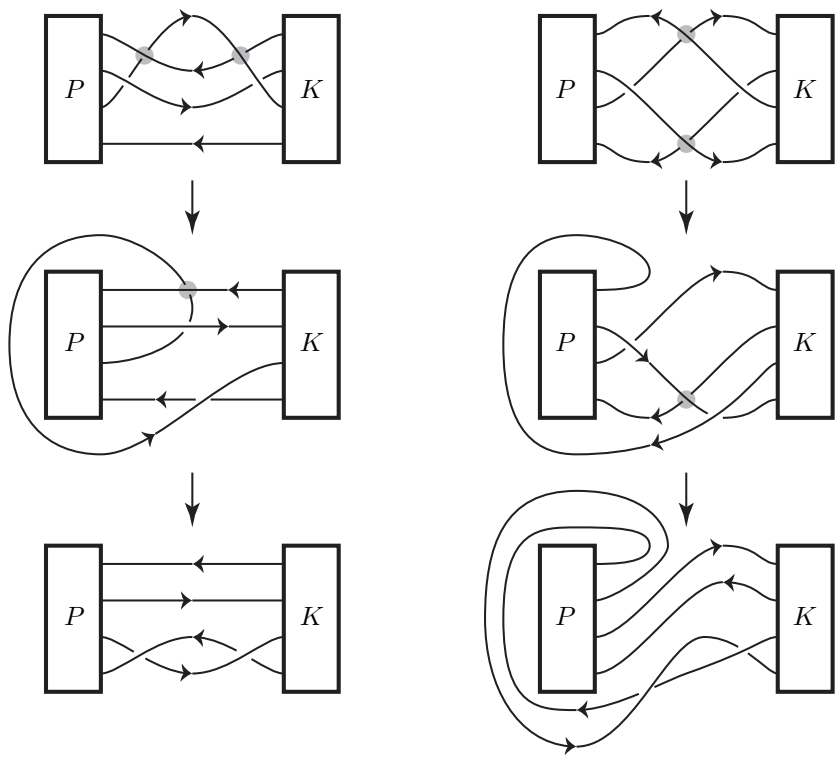
\begin{figure}
\labellist
	\small
	\pinlabel $P$ [b] at 41 65
	\pinlabel $K$ [b] at 158 65

	\pinlabel $P$ [b] at 41 200
	\pinlabel $K$ [b] at 158 200

	\pinlabel $P$ [b] at 41 325
	\pinlabel $K$ [b] at 158 325
	
	\pinlabel $P$ [b] at 283 65
	\pinlabel $K$ [b] at 401 65

	\pinlabel $P$ [b] at 283 200
	\pinlabel $K$ [b] at 401 200

	\pinlabel $P$ [b] at 283 325
	\pinlabel $K$ [b] at 401 325

\endlabellist

\centerline{\includegraphics[width=4.5in]{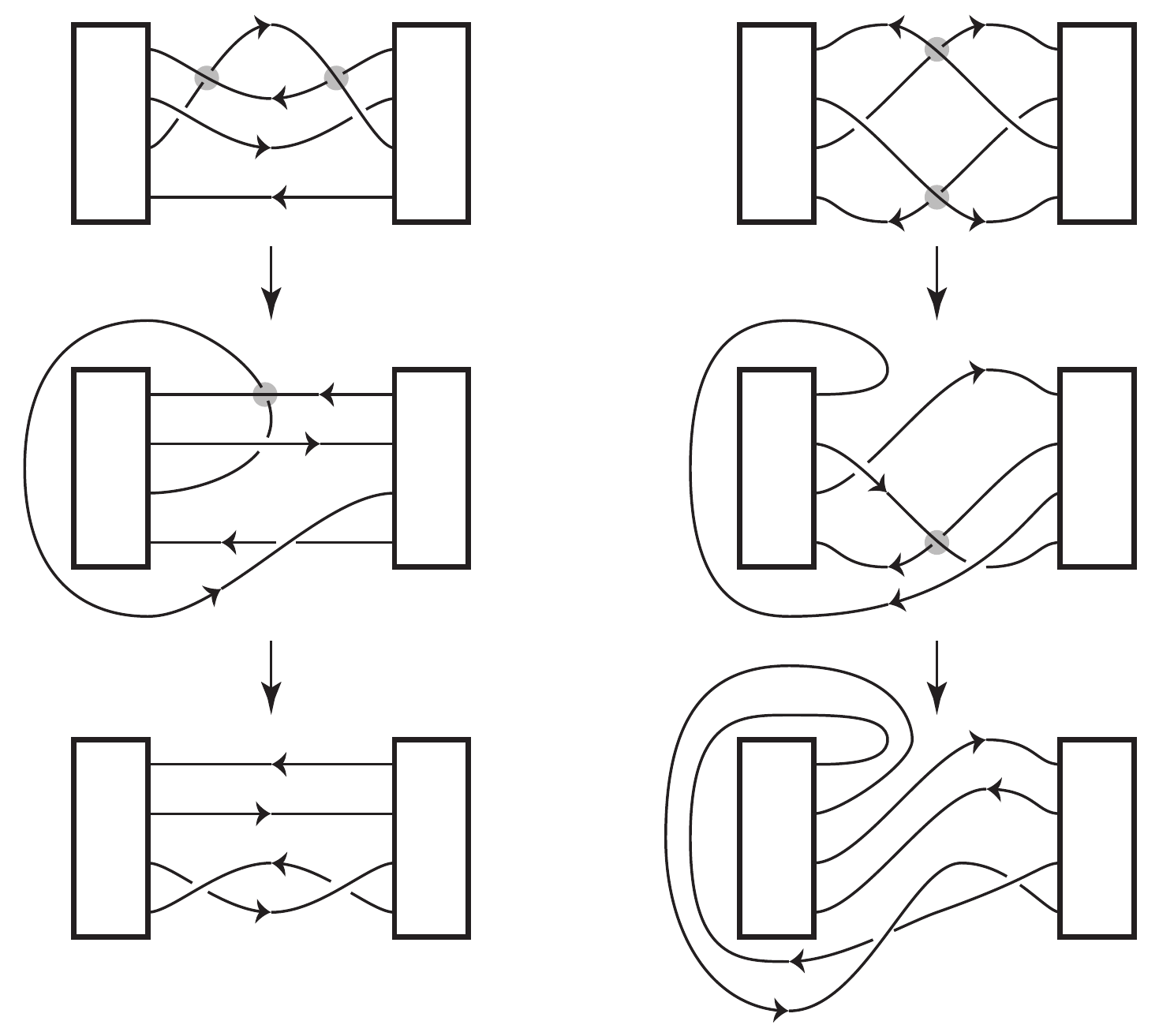}}
\caption{Eliminating a pair of negative crossings for a side double (left) and a split double (right) for a knot with first band positive. The shaded circles mark negative crossings.}
\label{fig:pos_corr}
\end{figure}

\begin{figure}
\labellist
	\small
	\pinlabel $P$ [b] at 48 65
	\pinlabel $K$ [b] at 166 65

	\pinlabel $P$ [b] at 48 200
	\pinlabel $K$ [b] at 166 200

	\pinlabel $P$ [b] at 48 335
	\pinlabel $K$ [b] at 166 335
	
	\pinlabel $P$ [b] at 291 65
	\pinlabel $K$ [b] at 409 65

	\pinlabel $P$ [b] at 291 200
	\pinlabel $K$ [b] at 409 200

	\pinlabel $P$ [b] at 291 335
	\pinlabel $K$ [b] at 409 335

\endlabellist

\centerline{\includegraphics[width=4.5in]{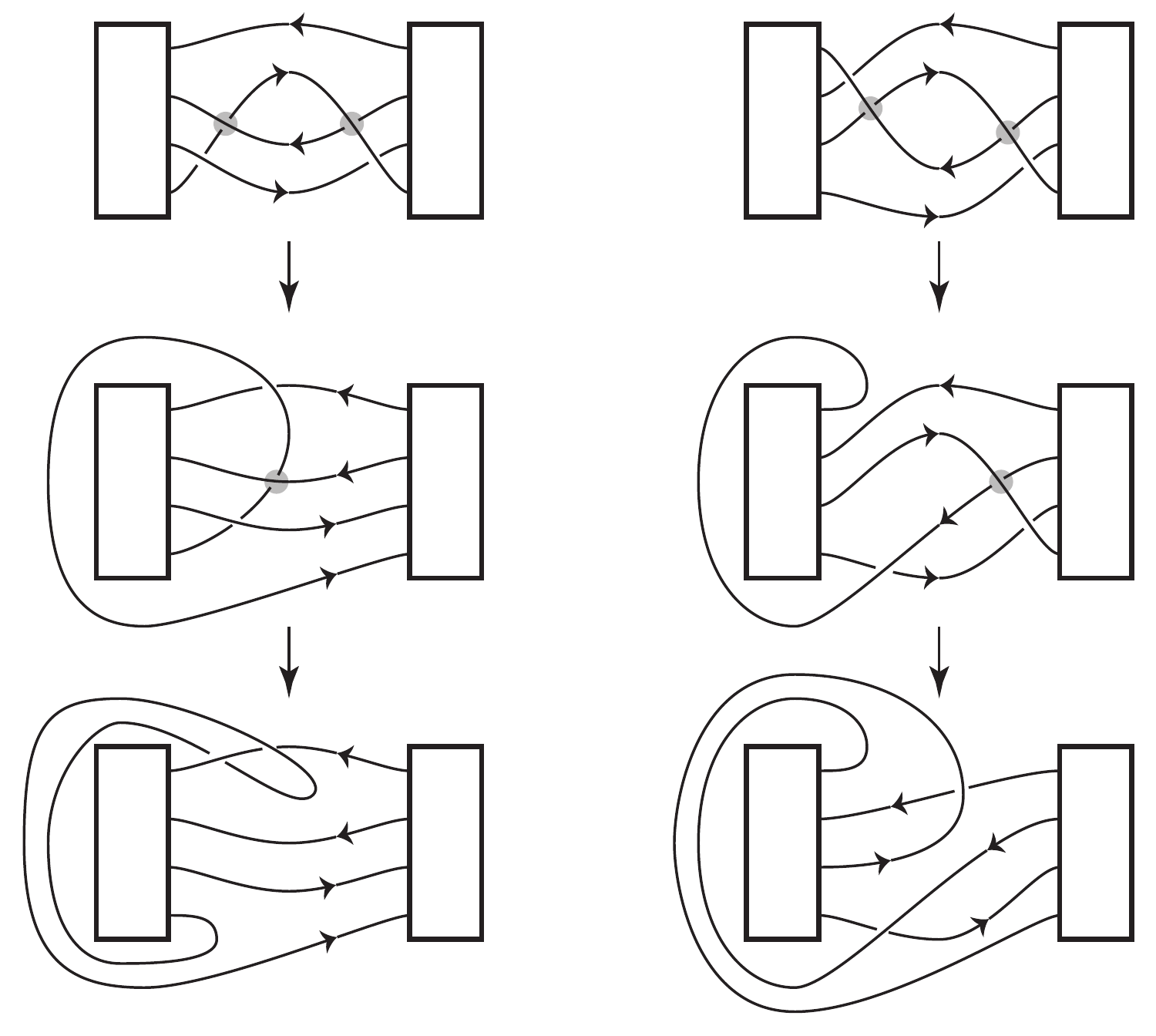}}
\caption{Eliminating a pair of negative crossings for a center band for a knot with first band negative. The two cases correspond to two different relative positions of the adjacent side bands. The shaded circles mark negative crossings.}
\label{fig:neg_corr}
\end{figure}

Suppose inductively that we have eliminated negative crossings up to the band $b_k$ while keeping the diagram in P4P form with the crossings of $b_k$ and those immediately to the left and right serving as the middle crossings.  We eliminate the two negative crossings of $b_k$ in four cases:
\begin{itemize}
\item $b_k$ is a pure side band (see Figure~\ref{fig:pos_corr}(a)),
\item $b_k$ is a split band (see Figure~\ref{fig:pos_corr}(b)),
\item $b_k$ is a center band with adjacent crossings on the same side (see Figure~\ref{fig:neg_corr}(a)), and
\item $b_k$ is a center band with adjacent crossings on opposite sides (see Figure~\ref{fig:neg_corr}(b)).
\end{itemize}
It is easy to see that after performing the isotopies depicted in Figures~\ref{fig:pos_corr} and \ref{fig:neg_corr}, the diagram is once again in P4P form with middle crossings around and including $b_{k+2}$.  

If the last band is negative, and hence has a single crossing, simply perform a Reidemeister II move to remove it.  Either way, at the end of this procedure, we are left with a positive diagram, hence proving the corollary.


\bibliographystyle{amsplain}
\bibliography{main}

\end{document}